\documentclass[twoside,12pt]{article}
\usepackage{latexsym,amsmath, amsfonts, graphics, epsf}%, epic
\usepackage{amsthm}
\usepackage{amssymb}
\usepackage{amsopn}
\usepackage{amscd}

\theoremstyle{plain}
\newtheorem{thm}{Theorem}[section]
\newtheorem{cor}[thm]{Corollary}
\newtheorem{lem}[thm]{Lemma}
\newtheorem{prop}[thm]{Proposition}

\newtheorem{conjecture}[thm]{Conjecture}

\theoremstyle{definition}

\newtheorem{remark}[thm]{Remark}
\newtheorem{example}[thm]{Example}
\newtheorem{defn}[thm]{Definition}

\usepackage{amssymb}

\newcommand{\bee}[1]{\begin{equation}\label{#1}}
\newcommand{\beq}[1]{\begin{eqnarray}\label{#1}}
\newcommand{\ene}{\end{equation}}
\newcommand{\eqe}{\end{eqnarray}}

\title{The Murnaghan-Nakayama rule and some virtual $S_n$ characters}
\author{ Amitai Regev \\Mathematics Department\\The Weizmann Institute\\Rehovot 76100\\Israel}

\begin{document}

\maketitle
%characters.tex

\bigskip

{\bf Abstract}. We construct certain virtual characters for the symmetric groups, then compute a
formula which calculates the values of these   virtual characters.

\section{Introduction}
Partitions are denoted here by $\lambda, \mu,\nu,$ etc., and a partition is identified with its Young diagram. As usual, $\lambda'$ denotes the conjugate partition of $\lambda$.
We write $\lambda\vdash n$ if $\lambda$ is a partition of $n$. When the characteristic of the base field is zero, the partitions
$\lambda\vdash n$ are in a one-to-one correspondence with the irreducible $S_n$-characters,
denoted $\chi^\lambda$,~\cite{jameskerber},~\cite{macdonald},~\cite{sagan}. An integer-combination
of irreducible characters is called a virtual character.

\medskip

We construct certain virtual characters $\psi_{\nu,n}$  of the symmetric group $S_n$. Here $\nu$
is a partition of $k$ where $k$ is much smaller that $n$.
These virtual characters are alternating sums of  certain irreducible  $S_n$ characters. The
main result here is, that the
values  $\psi_{\nu,n}(\mu)$ of these virtual characters on the partitions $\mu\vdash n$ are given by
one character formula. This is Theorem~\ref{v4} below. This formula shows that
the character tables of the symmetric groups satisfy many relations and identities.

\section{The virtual $S_n$-characters  $\psi_{\nu,n}$}\label{section2}
 Let $\nu=(\nu_1,\nu_2,\ldots)\vdash k$, and $n\ge 2k+2$.
First form $\nu^{(1)}=(\nu_1+n-k,\nu_2,\nu_3,\ldots)$. This is the diagram $\nu$ with $n-k$ additional boxes attached to its first row.
Now pull these added $n-k$ boxes  down and left around the diagram $\nu$ as follows.
Think of these $n-k$ boxes as a train pulled by the first (i.e.~left) cell.
And pull it down and left around $\nu$
so that the result is a partition (containing $\nu$).
We call this process "going around $\nu$", see Example~\ref{v2}.
This process, which
is analized in Section~\ref{section2.1}, yields an
ordered sequence of partitions of $n$: $\nu^{(1)},\nu^{(2)},\ldots\vdash n$.

\medskip
For example, verify that $\nu^{(2)}=(n-k+\nu_2-1,\nu_1+1,\nu_3,\nu_4,\ldots )$.
 Note that since $n\ge 2k+2$ and $\nu\vdash k$, it follows that $n-k+\nu_2-1\ge \nu_1+1$, and therefore $\nu^{(2)}$ is indeed a partition.

\begin{defn} (The virtual character $\psi_{\nu,n}$.)
Given $\nu\vdash k$ and $n\ge k$ (we usually require that $n\ge 2k+2$),
with the partitions of $n$ obtained by going around $\nu$,
these partitions are ordered as first, second, third, etc. We define
$\psi_{\nu,n}$  to be the alternating sum of the corresponding irreducible $S_n$ characters.

\end{defn}

\begin{example}\label{v2}
1.

\medskip
Going around the empty diagram $\nu=\emptyset$ we get the following sequence of diagrams:

\medskip

$(n)\to (n-1,1)\to (n-2,1^2)\to\cdots\to (2,1^{n-2})\to (1^n)$.

\medskip

Thus $$\psi_{\emptyset,n}= \sum_{j=0}^{n-1}(-1)^j\chi^{(n-j,1^j)}.
 $$

\medskip
2.~~Going around $\nu=(1)$. Here $k=1$,  so $n\ge 4$. Get the diagrams

\medskip

$(n)\to (n-2,2)\to (n-3,2,1)\to(n-4,2,1^2) \to\cdots\to (3,2,1^{n-5})\to(2,2,1^{n-4})\to (1^n)$.

\medskip
Therefore
$$
 \psi_{(1),n}=\chi^{(n)}+\sum_{j=0}^{n-4}(-1)^{j+1}\chi^{(n-2-j,2,1^j)} + (-1)^n\chi^{(1^n)}.
 $$

\medskip
3.~~Going around $(2)$. Here $k=2$ so $n\ge 6$.
Get the partitions

\medskip
$(n)\to(n-3,3)\to (n-4,3,1)\to (n-5,3,1^2)\to (n-6,3,1^3)\to\cdots\to (4,3,1^{n-7})\to \\
\to (3,3,1^{n-6})  \to(2^2,1^{n-4}) \to (2,1^{n-2})  $.

\medskip
Thus
\begin{eqnarray}\label{formula5}
 \psi_{(2),n}=\chi^{(n)}+\sum_{j=0}^{n-6}(-1)^{j+1}\chi^{(n-3-j,3,1^{j})} +
 (-1)^{n-1}\chi^{(2,2,1^{n-4})}+(-1)^n\chi^{(2,1^{n-2})}.
 \end{eqnarray}
\end{example}
\subsection{The general construction}\label{section2.1}
\begin{defn}\label{height1}
1. ~Let $n\ge k$, $\nu\vdash k$, and $\eta\vdash n$.
We shall assume that $n\ge 2k+2$.
Assume $\nu \le\eta$ and that $S$ is a part of the rim of $\eta$
such that $\eta\setminus S=\nu$, then we write $\eta=\nu*S$. Let $h(S)$ denote the height of $S$.

\medskip
2. ~Given such  $\eta=\nu*S$, we say that "$S$ covers $\nu$" if $\nu'_1\le h(S)$. Otherwise $S$ covers
only an upper part of $\nu$.

\medskip
3. ~Given $\nu\vdash k,$ we start by constructing $\nu*S_1$. Here $S_1 $ is the one row
of length $n-k$, added to the first row
of $\nu$. Clearly, $h(S_1)=1$. Continue and construct the sequence of partitions
$\nu*S_1$, $~\nu*S_2$, $~\ldots ~,~\nu*S_{n-k}.$
Lemma~\ref{height2} shows that  $h(S_j)=j$
for $1\le j\le n-k$, and the process stops at $j=n-k$, namely after $n-k$ steps.

\medskip
4. ~Define the virtual $S_n$ character $\psi_{\nu,n}$ as follows:
\begin{eqnarray}\label{psi1}
\psi_{\nu,n}:=\sum_{j=1}^{n-k} (-1)^{j+1}\chi^{\nu*S_j}.
\end{eqnarray}

5.~ We say that $\nu*S$ {\it "has a tail"} if $\nu*S=\mu=(\mu_1,\mu_2,\ldots)$ where
$\mu_1\ge \nu_1+1$.
\end{defn}

The following unique-decomposition lemma is crucial here.
\begin{lem}\label{uniqueness1}
Let $n\ge 2k+2$. For $i=1,2$ let $\nu^{(i)}\vdash k$, and let $S^{(i)}$ be a rim of
$\nu^{(i)}*S^{(i)}$ of length $n-k$.
$$
 \mbox {If}\quad\nu^{(1)}*S^{(1)}=\nu^{(2)}*S^{(2)}\quad\mbox{then}\quad \nu^{(1)}=\nu^{(2)}
\quad\mbox{and}\quad S^{(1)}=S^{(2)}.
$$
\end{lem}
Note that Example~\ref{example1} below shows that the condition $n\ge 2k +2$ is necessary.
\begin{proof}
Denote $\eta=\nu^{(1)}*S^{(1)}=\nu^{(2)}*S^{(2)}$.
There are three cases to consider.

\medskip
Case 1: $S^{(1)}$ covers an upper part of $\nu^{(1)}$ (and has a North-East tail).

\medskip
case 2: The conjugate of case 1.

\medskip
Case 3: $S^{(1)}$ completely covers $\nu^{(1)}$ (and might have a North-East and/or
South-East tails).

\medskip
Case 1: Here the rim of $\nu^{(1)}*S^{(1)}$ contains $S^{(1)}$ and possibly an additional part which
is a part of $\nu^{(1)}$. Assume $\eta=\nu^{(1)}*S^{(1)}=\nu^{(2)}*S^{(2)}$ with
$\nu^{(1)}\ne \nu^{(2)}$. Then $\nu^{(2)}$ contains a cell of $S^{(1)}$. This cell splits
 $S^{(1)}$ into two parts:  the  North-East part of $S^{(1)}$ and the
South-East part (which contains
the South-East part of $S^{(1)}$). Let $\tilde S$ denote the South-East part of $S^{(1)}$ together with the lower part of the
rim of $\nu^{(2)}$. We need to show that neither the  North-East part of $S^{(1)}$, nor  $\tilde S$
can be    $S^{(2)}$.

\medskip

The  North-East part is properly contained in  $S^{(1)}$  hence has length
strictly less than $|S^{(1)}|=n-k$, thus this part cannot be $S^{(2)}$.

\medskip

Since we are in case 1, the  part $\tilde S$ is a part of
$\nu^{(2)}$. But $\nu^{(2)}\vdash k$, hence  $\tilde S$ is of length
$|\tilde S|\le k<n-k$, hence  this part cannot be $S^{(2)}$ either.

\medskip
So case 1 is impossible. By conjugation, case 2 is also impossible.

 \medskip
  Case 3. The argument here is similar: A cell of $\nu^{(2)}$ on $S^{(1)}$ splits
  the rim of $\nu^{(1)}*S^{(1)}$ (namely $S^{(1)}$) into
  two parts, each of length strictly less than $n-k$, hence neither can be $S^{(2)}$.
  Thus case 3 is also impossible, and the proof follows.
\end{proof}

\begin{lem}\label{height2}
Let $\nu\vdash k$ and let $n\ge 2k+2.$
As in Definition~\ref{height1}, construct the sequence of partitions
$\nu*S_1$, $~\nu*S_2$, $~\ldots ~,~\nu*S_{n-k}$. Then $h(S_j)=j$ for $1\le j\le n-k$.
\end{lem}

Here also, Example~\ref{example1} shows that the condition $n\ge 2k +2$ is necessary.

\begin{proof}
Analogue to the height $h(S)$, we also define the width $w(S)$ (which is the height of the
conjugate rim $S'$). By projecting $S$ on the axes (see also (1.7) in~\cite{macdonald}), it follows that
\begin{eqnarray}\label{height.width}
|S|=h(S)+w(S)-1.
\end{eqnarray}

\medskip

Note that $n-k\ge \nu_1+\nu'_1+1$. Indeed, $k\ge \nu_1+\nu'_1-1$, so $n-k\ge k+2
\ge \nu_1+\nu'_1+1$.

%And $S$ "covers" $\nu$ if  $\nu*S$ has a tail and also $h(S)\ge \nu'_1$.

\medskip
As in Definition~\ref{height1}.2,
construct the sequence $\nu*S_1, ~\nu*S_2,\ldots$.  For the first  $\nu_1'+1$ steps,
by construction and induction, $h(S_j)=j,~1\le j\le \nu_1'+1$
(since $|S|=n-k$, while the number of boxes in $S$ except those in the first row of $\nu*S$
is $\le \nu_1+\nu'_1\le k$).
After $\nu_1'+1$ steps, $\nu$ is "covered", with $\nu_1+\nu'_1+1$ out of the $n-k$ cells are
covering $\nu$. Since $n-k\ge \nu_1+\nu'_1+1$,
at that stage
there still is a "tail" of added boxes in the first (i.e.~top) row of $\nu*S_j,$ $j=\nu_1'+1$.

\medskip
For the remaining steps we consider the conjugate construction. Then, applying~\eqref{height.width} together with the above argument,
the proof follows.
\end{proof}

\begin{example}\label{example1}
{\bf Counter examples} when $n\not\ge 2k+2$.

\medskip
1. ~Let $k=3,~n=6,~\eta=(3,3),~\nu^{(1)}=(3), ~\nu^{(2)}=(2,1)$. Let $S^{(i)}$, $i=1,2$, satisfy
$$(3,3)=\nu^{(1)}*S^{(1)}=\nu^{(2)}*S^{(2)}.$$
Then clearly $S^{(1)}\ne S^{(2)}$, as well as $\nu^{(1)}\ne \nu^{(2)}$.

\medskip

2. ~Again let $k=3,~n=6$. Let $\nu=(3)$ and construct the sequence of partitions\\
$\nu*S_1$, $~\nu*S_2$, $~\nu*S_3\ldots $ Then $h(S_2)=1\ne 2$.

\end{example}

We proceed with the general case.

\begin{cor}\label{height3}
Let $\nu\vdash k$, and let $n\ge 2k+2$, then
$$
\psi_{\nu,n}=\sum_{S,\;\nu*S\vdash n}(-1)^{h(S)+1}\chi^{\nu*S}=\sum_{j=1}^{n-k}(-1)^{h(S_j)+1}\chi^{\nu*S_j}.
$$
\end{cor}
\begin{proof}
This follows from Equation~\eqref{psi1} (which defines $\psi_{\nu,n}$) and from Lemma~\ref{height2}.

\end{proof}

\section{A formula for $\psi_{\nu,n}(\mu)$}
Our aim is to prove the following formula.

\begin{thm}\label{v4}

Let $\nu=(\nu_1,\nu_2,\ldots )\vdash k$, $n\ge 2k+2$ . Let
$\mu=(\mu_1,\mu_2,\ldots )\vdash n$, and denote $\bar\mu=(\mu_2,\mu_3,\ldots)$,
so $\mu_1=n-k$ if and only if $\bar\mu\vdash k$. Then

\begin{eqnarray}\label{general.conjecture}
\psi_{\nu,n}(\mu)=\begin{cases}\chi^\nu(\bar\mu)\cdot(n-k) & if  ~\mu_1= n-k
\\
0 & if ~\mu_1\ne n-k\\
\end{cases}
\end{eqnarray}
\end{thm}
The proof is given below.

\medskip
Let  $n\ge 2k+2$ and let $\nu\vdash k$, so $(n-k,\nu)$ is a partition of $n$. Let $\lambda\vdash n$ and assume $\chi^\lambda_{(n-k,\nu)}\ne 0$, then,
by the Murnaghan-Nakayama (M-N) rule,  $\lambda $ can be written, probably in several ways, as
$\lambda=\rho*S$ where $\rho\vdash k$, $S$ is part of the rim of $\lambda$ and $|S|=n-k$. By
Lemma~\ref{uniqueness1} this decomposition, with $|\rho|=k$, is unique, hence we can write
$\lambda\leftrightarrow (\rho_\lambda ,S_\lambda).$
Again by the M-N rule,with $\lambda=\rho*S$,
\begin{eqnarray}\label{basic2}
\chi^\lambda_{(n-k,\nu)}=(-1)^{h(S)+1}\chi^{\rho}_\nu=
(-1)^{h(S_\lambda)+1}\chi^{\rho_\lambda}_\nu.
\end{eqnarray}

Lemma~\ref{uniqueness1} allows us to prove the following formula.
% (now -- just a conjecture).

\begin{prop}\label{usemn}
Let $\nu\vdash k$ and let $n\ge 2k+2$. Then
$$
\sum_{\lambda\vdash n}\chi^\lambda_{(n-k,\nu)}\chi^\lambda=\sum_{\rho\vdash k}
\chi^\rho_\nu \cdot \psi_{\rho,n}.
$$
\end{prop}

\begin{proof}
%We present an argument which applies the conjectured Lemma~\ref{uniqueness1}.

\medskip
We just saw that
$$
\{\lambda\vdash n\}=\{\rho*S\vdash n\mid \rho\vdash k,~~|S|=n-k\}\cup
\{\lambda\vdash n\mid\chi^\lambda_{(n-k,\nu)}=0\},
$$
and by Lemma~\ref{uniqueness1}  we have the bijection
$$
\{\rho*S\vdash n\mid \rho\vdash k,~~|S|=n-k\}\longleftrightarrow
\{(\rho,S)\mid \rho\vdash k,~~|S|=n-k\quad\mbox{and}\quad \rho*S\vdash n\}.
$$
We denote
$$
A_{k,n}=\{(\rho,S)\mid \rho\vdash k,~~|S|=n-k\quad\mbox{and}\quad \rho*S\vdash n\}.
$$
By~\eqref{basic2}
$$
\sum_{\lambda\vdash n}\chi^\lambda_{(n-k,\nu)}\chi^\lambda=
\sum_{(\rho,S)\in A_{k,n}}(-1)^{h(S)+1}\chi^\rho_\nu\cdot\chi^\lambda=
$$
$$
\sum_{\rho\vdash k}~~\sum_{S,\;\rho*S\vdash n}(-1)^{h(S)+1}\chi^\rho_\nu\cdot\chi^{\rho*S}=
$$
$$
~~~~~~~~~~~~~
\sum_{\rho\vdash k}\chi^\rho_\nu~\sum_{S,\;\rho*S\vdash n}(-1)^{h(S)+1}\chi^{\rho*S}=
\sum_{\rho\vdash k}\chi^\rho_\nu\cdot\psi_{\rho,n}.
$$
The last equality applied corollary~\ref{height3}.
\end{proof}

\begin{remark}\label{u33}
\end{remark}
Let $d=|\{Par(k)\}|.$
Note that in matrix form, Proposition~\ref{usemn} can be written as follows:
\begin{eqnarray}\label{usemn2}
\left[\chi^\rho_\nu\right]\left[ \psi_{\rho,n}\right]=
\left[\sum_{\lambda\vdash n}\chi^\lambda _{(n-k,\nu)}\cdot\chi^\lambda \right].
\end{eqnarray}
Here
$\left[\chi^\rho_\nu\right]$ is the $d\times d$ character table of $S_k$, and $\left[ \psi_{\rho,n}\right]$
is a column of height $d$. Of course,
the locations of the entries of
both $\left[\chi^\rho_\nu\right]$ and $\left[ \psi_{\rho,n}\right]$
depend on how we order $Par(k)$.
Applying~\eqref{usemn2} on $\mu\vdash n$ we get
\begin{eqnarray}\label{usemn3}
\left[\chi^\rho_\nu\right]\left[ \psi_{\rho,n}(\mu)\right]
=\left[\sum_{\lambda\vdash n}\chi^\lambda _{(n-k,\nu)}\cdot\chi^\lambda_\mu \right]
\end{eqnarray}

Recall the classical column-orthogonality-relations for the $S_n$ characters.
\begin{thm}\label{v5} (The column-orthogonality-relations)
\begin{eqnarray}\label{columns1}
\sum_{\lambda\vdash n}\chi^\lambda_\eta\chi^\lambda_\mu=\begin{cases}|Z_{S_n}(\eta)| & if ~\mu=\eta
\\
0 & if~ \mu\ne \eta,\\
\end{cases}
\end{eqnarray}
where $\lambda,\eta,\mu\vdash n$, and $Z_{S_n}(\eta)$
 is the centralizer of $\eta$  in $S_n$ (i.e~the centralizer of $\pi\in S_n$ with
cycle structure $\eta$).
Let $K=K^{(n)}=\left[ \chi_{\eta}^{\theta}\right]$ denote the character table of $S_n$, then
~\eqref{columns1} can be written as
\begin{eqnarray}\label{columns2}
 KK^T=diag(|Z_{S_n}(\eta)|,\mid \eta\vdash n).
  \end{eqnarray}

\end{thm}
Then  numbers $|Z_{S_n}(\eta)|$ are calculated through the following well known formula.
\begin{thm}\label{v1}
Let $\eta\vdash n$, $\eta=(1^{m_1},2^{m_2},\ldots )$, and let $Z_{S_n}(\eta)$ denote the
centralizer in $S_n$ of  $\sigma\in S_n$
where $\eta$ is the cycle structure of $\sigma$. Then
$$
|Z_{S_n}(\sigma)|=|Z_{S_n}(\eta)|=\prod_i (i^{m_i}\cdot m_i!).
$$
\end{thm}

\begin{remark}\label{centralizer1}
Let $\nu\vdash k$, $n\ge 2k+2$. Then
$(n-k,\nu)$ is a partition (of $n$), and
$$|Z_{S_n}(n-k,\nu)|=|Z_{S_k}(\nu)|\cdot (n-k).$$
\end{remark}
This follows from Theorem~\ref{v1} since $n-k$ is strictly larger than any component of $\nu$.
Together with~\eqref{columns1} this implies

\begin{cor}\label{centralizer2}
Let $\nu\vdash k$, $n\ge 2k+2$, $\mu=(\mu_1,\mu_2,\ldots)\vdash n$ and let $\bar \mu=(\mu_2,\mu_3,\ldots)$. Then
\begin{eqnarray}\label{columns12}
\sum_{\lambda\vdash n}\chi^\lambda_{(n-k,\nu)}\chi^\lambda_\mu=
\begin{cases}|Z_{S_k}(\nu)|\cdot(n-k) & if ~\bar\mu=\nu
\\
0 & if~ \mu\ne \eta.\\
\end{cases}
\end{eqnarray}
\end{cor}
\subsection{The proof of Theorem~\ref{v4} }
Apply~\eqref{usemn3} and Corollary~\ref{centralizer2}: With
$\rho$ and $\nu$ denoting partitions of $k$ and $\mu$ partitions of $n$
 we have
$$
\left[\chi^\rho_\nu\right]\left[ \psi_{\rho,n}(\mu)\right]
=\left[\sum_{\lambda\vdash n}\chi^\lambda _{(n-k,\nu)}\cdot\chi^\lambda_\mu \right]
=\left[\begin{cases}|Z_{S_k}(\nu)|\cdot(n-k) & if ~\bar\mu=\nu
\\
0 & if~ \bar\mu\ne \nu.\\
\end{cases}
\right]=
$$
\begin{eqnarray}\label{usemn31}
~~~~~~~~~~~~~~~~~~~~~~~~~~~~~~~~~~~~~~~~~~~~~~~~~~~~~~~~~
=\left[\begin{cases}|Z_{S_k}(\nu)| & if ~\bar\mu=\nu
\\
0 & if~ \bar\mu\ne \nu.\\
\end{cases}
\right]\cdot(n-k)
\end{eqnarray}
Restrict now to partitions $\mu=(n-k,\bar\mu)$ (so $\mu$ is given by $\bar\mu$), then~\eqref{usemn31} becomes

$$
\left[\chi^\rho_\nu\right]\left[ \psi_{\rho,n}(n-k,\bar\mu)\right]
=\left[\sum_{\lambda\vdash n}\chi^\lambda _{(n-k,\nu)}\cdot\chi^\lambda_{(n-k,\bar\mu)} \right]
=\left[\begin{cases}|Z_{S_k}(\nu)|\cdot(n-k) & if ~\bar\mu=\nu
\\
0 & if~ \bar\mu\ne \nu.\\
\end{cases}
\right]=
$$

\begin{eqnarray}\label{usemn311}
~~~~~~~~~~~~~~~~~~~~~~~~~~~~~~~~~~~~~~~~~~~~~~~~~~~~~~~~~
=\left[\begin{cases}|Z_{S_k}(\nu)| & if ~\bar\mu=\nu
\\
0 & if~ \bar\mu\ne \nu\\
\end{cases}
\right]\cdot(n-k).
\end{eqnarray}

Note that each side of~\eqref{usemn311} is a column of height $d=|\{Par(k)\}|$, parametrized
by the partitions $\bar\mu\in Par(k)$. Thus, if $\nu=\bar\mu^{(j)}$ is the $j$th element in
$Par(k)$ then the transpose of the corresponding column of~\eqref{usemn311} is
$$
(0,\ldots,0,|Z_{S_k}(\bar\mu^{(j)})|,0,\ldots,0)\cdot (n-k).
$$
The columns $\left[ \psi_{\rho,n}(n-k,\bar\mu)\right]$ of height $d$
is parametrized by $\rho\in Par(k)$, with $\bar\mu$ fixed.
From these columns we form the $d\times d$ matrix
$$
M=\left[ \psi_{\rho,n}(n-k,\bar\mu^{(1)}),\psi_{\rho,n}(n-k,\bar\mu^{(2)}),\ldots,
 \psi_{\rho,n}(n-k,\bar\mu^{(d)})\right].
$$
Also denote
 $K=\left[\chi^\rho_\nu\right]$ and $D=diag( |Z_{S_k}(\bar\mu^{(1)})|,\ldots,|Z_{S_k}(\bar\mu^{(d)})| )$.
By~\eqref{usemn311} we get
$$
\left[\chi^\rho_\nu\right] M=KM=D\cdot(n-k)
=diag( |Z_{S_k}(\bar\mu^{(1)})|,\ldots,|Z_{S_k}(\bar\mu^{(d)})| )\cdot(n-k),
$$
and by~\eqref{columns2}
$$
KK^T=D=diag( |Z_{S_k}(\bar\mu^{(1)})|,\ldots,|Z_{S_k}(\bar\mu^{(d)})| ).
$$

Thus we have $KM=KK^T\cdot(n-k).$

\medskip
Note that $K$ is the character table of $S_k$, hence is invertible.
Left cancelation of $K$ implies that $M=K^T\cdot (n-k)$, so $M$ = (the character table of $S_k)\cdot(n-k)$.
This completes the proof of Theorem~\ref{v4}.

\subsection{Applications}

\begin{example}
1.~In Example~\ref{v2}.1 we saw that $\psi_{\emptyset,n}= \sum_{j=0}^{n-1}(-1)^j\chi^{(n-j,1^j)}.$
Theorem~\ref{v4} with $k=0$ then implies that
\begin{eqnarray}\label{general.conjecture}
  \psi_{\emptyset,n}(\mu)= \sum_{j=0}^{n-1}(-1)^j\chi^{(n-j,1^j)}(\mu) =\begin{cases}n & if  ~\mu= (n)
\\
0 & if ~\mu\ne (n)\\
\end{cases}
\end{eqnarray}

2.~Similarly, by Example~\ref{v2}.2,
$
 ~\psi_{(1),n}=\chi^{(n)}+\sum_{j=0}^{n-4}(-1)^{j+1}\chi^{(n-2-j,2,1^j)} + (-1)^n\chi^{(1^n)},
 $
and we get
$$
\psi_{(1),n}(\mu)=~~~~~~~~~~~~~~~~~~~~~~~~~~~~~~~~~~~~~~~~~~~~~~~~~~~~~~~
~~~~~~~~~~~~~~~~~~~~~~~~~~~~~~~~~~~~~~~~~~~~~~~~~~~~~~~~~~~~
$$

\begin{eqnarray}\label{general.conjecture2}
  =\chi^{(n)}(\mu)+\sum_{j=0}^{n-4}(-1)^{j+1}\chi^{(n-2-j,2,1^j)}(\mu) + (-1)^n\chi^{(1^n)}(\mu) =\begin{cases}n-1 & if  ~\mu= (n-1,1)
\\
0 & if ~\mu\ne (n-1,1)\\
\end{cases}
\end{eqnarray}

Clearly, there are infinitely many identities that can be deduced this way.

\medskip
3.~ Consider the first column $\{f^\lambda\mid \lambda\vdash n\}$ in the charcter table of
$S_n.$ Let $n\ge 2k+2$, fix some $\nu\vdash k$, then construct the sequence of partitions of $n$:
$\nu*S_1,\nu*S_2,\ldots$. Finally, form the corresponding alternating sum, then always
\begin{eqnarray}\label{v22}
\sum_{j=1}^{n-k}(-1)^jf^{\nu*S_j}=0.
\end{eqnarray}
This follows from Theorem~\ref{v4}, since this corresponds to $\mu=(1^n)$, so $\mu_1=1\ne n-k$.
When $k=0$ this is the well known identity
\begin{eqnarray}\label{general.conjecture02}
\sum_{j=1}^n(-1)^j{n\choose j}=0,
\end{eqnarray}
see~\eqref{general.conjecture}. Thus,~\eqref{v22} can be seen as a generalization
of~\eqref{general.conjecture02}.

\end{example}

\begin{remark}

In~\cite{regev} a formula is proved for the values
$$
\sum_{i\ge 0} \chi^{(n-i,1^i)}(\mu),\quad \mu\vdash n.
$$
Of course, here we deduced a formula for the values for the corresponding alternating sum
$$
\sum_{i\ge 0} (-1)^i\chi^{(n-i,1^i)}(\mu),\quad \mu\vdash n.
$$

Adding, we get a formula for
$$
\sum_{i~even} \chi^{(n-i,1^i)}(\mu),\quad \mu\vdash n,
$$
hence also for
$$
\sum_{i~odd} \chi^{(n-i,1^i)}(\mu),\quad \mu\vdash n.
$$
We leave the details for the reader.
\end{remark}

\section{Final remarks}
\begin{remark}\label{(1)}
Let $n\ge 4$ and let $A_n$ be the set of partitions obtained by walking around $\nu=(1)$:
$$
A_n=\{(n),(1^n)\}\cup\{(r,2,1^{n-2-r})\mid 2\le r\le n-2\}.
$$
Then
\begin{eqnarray}\label{(1)2}
\sum_{\lambda\in A_n}f^\lambda=(n-4)\cdot 2^{n-2}+4.
\end{eqnarray}

\begin{proof}
\medskip
Let $\lambda=(r,2,1^{n-2-r})$, then
$$
f^\lambda=\frac{n!}{(r-2)! (n-2-r)! r (n-r) (n-1)}=\frac{n(n-2)(n-3)}{r(n-r)}{n-4\choose r-2}.
$$
Thus, omitting $\lambda\in \{(n),(1^n)\}$, we need to show that
\begin{eqnarray}\label{wz}
\sum_{r=0}^{n-4}\frac{n(n-2)(n-3)}{r(n-r)}{n-4\choose r-2} =(n-4)\cdot 2^{n-2}+2.
\end{eqnarray}
This can be proved by the Zeilberger Algorithm~\cite{doron}, and is implemented in Maple (function SumTools[Hypergeometric][Zeilberger]).
\end{proof}
\end{remark}
Here is a direct proof.
\begin{proof}
[Zeilberger]  Note that
$$
\frac{1}{r(n-r)}=\frac{1}{n}\left( \frac{1}{r}+\frac{1}{n-r}\right),
$$
 so the sum in~\eqref{wz} equals
$$
(n-2)(n-3)\sum_{r=2}^{n-2}\frac{1}{r}{n-4\choose r-2}+
(n-2)(n-3)\sum_{r=2}^{n-2}\frac{1}{n-r}{n-4\choose r-2}.
$$
By symmetry these two summands are equal, so this equals
$$
2(n-2)(n-3)\sum_{r=2}^{n-2}\frac{1}{r}{n-4\choose r-2}=~~~~~~~~~~~~~~~~~~~~~~~~~~~~~~~~~~~
$$
$$
2(n-2)(n-3)\sum_{r=2}^{n-2}\left(\int_0^1 x^{r-1} dx\right){n-4\choose r-2}=
$$
$$
2(n-2)(n-3)\sum_{r=2}^{n-2}\left(\int_0^1 x^{r-1} dx{n-4\choose r-2}\right)=
$$
$$
2(n-2)(n-3)\sum_{r=0}^{n-4}\left(\int_0^1 x^{r+1} dx{n-4\choose r}\right)=
$$
$$
2(n-2)(n-3)\int_0^1 x(1+x)^{n-4} dx=
$$
$$
2(n-2)(n-3)\left( \int_0^1 (1+x)^{n-3}dx- \int_0^1 (1+x)^{n-4}dx  \right)=
$$
$$
=(n-4) 2^{n-2}+2.
$$
\end{proof}

Here is a second, combinatorial proof [Zeilberger].

\medskip
Let $f(n)$ be the number of standard Young tableaux of shape $(r,2,1^{n-2-r})$ for some
$2\le r\le n-2$, and look at the location of $n$.

\medskip
Case 1: It is at the rightmost cell of the top row. Deleting it gives something counted by $f(n-1)$.

\medskip
Case 2. The conjugate case, also yielding  $f(n-1)$.

\medskip
Case 3: $n$ is in the $(2,2)$ cell. Deleting it gives a tableau of strict hook shape  (namely,
not $(n-1)$ nor $(1^{n-1})$). Clearly, the number of these tableaux is $2^{n-1}-2$.

\medskip
So we have the recurrence $f(n)=2f(n-1) +2^{n-1}-2$. Since $f(4)=2$, it follows by induction that
$f(n)=(n-4) 2^{n-2}+2$.

\subsection{The polynomials $p_{\nu,n}(t)$}
%\begin{remark}
%The polynomial $p_{\nu,n}(t)$.

\medskip
Recall Equation~\eqref{v22}, replace minus $1$ by $t$ and get the polynomials
\begin{eqnarray}\label{v222}
p_{\nu,n}(t)=\sum_{j=1}^{n-k}f^{\nu*S_j}\cdot t^j.
\end{eqnarray}
In the case $\nu$ is empty we get $p_{\emptyset,n}(t)=(t+1)^{n-1}$. Denote
$p_{(1),n}(t)=p_n(t)$ in the case $\nu=(1)$. By direct computations we got the following polynomials,
for $1\le n\le 9$.

\newpage
\medskip
$p_4(t)=(t+1)^2$

\medskip
$p_5(t)=(t+1)(t^2+4t+1)$

\medskip
$p_6(t)=(t+1)^2(t^2+7t+1)$

\medskip
$p_7(t)=(t+1)(t^4+13t^3+22t^2+13t+1)$

\medskip
$p_8(t)=(t+1)^2 (t^4+18t^3+27t^2+18t+1)$

\medskip
$p_9(t)=(t+1)(t^6+26t^5+79t^4+110t^3+79t^2+26t+1)$

\medskip
$p_{10}(t)=(t+1)^2(t^6+33t^5+93t^4+131t^3+93t^2+33t+1)$

\bigskip
By~\eqref{v22} $p_n(t)$ is divisible by $t+1$ for all $n$.

%$\Big1\vert_0^1$

\begin{conjecture}
We conjecture that when $n$ is even, the highest power of $t+1$ which divides $p_n(t)$ is 2,
while when $n$ is odd, the highest power on $t+1$ which divides $p_n(t)$ is 1.

\medskip
One also conjectures positiveness and unimodality on the coefficients of these polynomials.

Additional conjecture is as follows. Denote $p_{2k+1}(t)=(t+1)\cdot q_{2k+1}(t)$, then
$q_{2k+1}(-1)=-2$.

\end{conjecture}

%\end{remark}


\begin{thebibliography}{99}

\bibitem{jameskerber} G. James and A. Kerber, The Representation Theory of the Symmetric Group, Encyclopedia
of Math. Vol. 16, Addison-Wesley Publishing Company.

\bibitem{macdonald} I. G. Macdonald, Symmetric Functions and Hall Polynomials, Second
Edition (1995).

\bibitem{regev} A. Regev, Lie superalgebras and some characters of $S_n$, to appean in
Israel J. Math (2013).

\bibitem{sagan} B. E. Sagan, The Symmetric Group, Graduate Text in Mat. Second Edition (2000).

\bibitem{doron} D. Zeilberger, The method of creative telescoping, J. Symbolic Computation 11,
195-204 (1991).





\end{thebibliography}
\end{document}